\documentclass[a4paper]{article}

\usepackage{amsfonts}
\usepackage{amsmath}
\usepackage{amssymb}
\usepackage{amsthm}
\usepackage[UKenglish]{babel}
\usepackage{Baskervaldx}
\usepackage{esint}
\usepackage{bm}
\usepackage{booktabs}
\usepackage{cite}
\usepackage[T1]{fontenc}
\usepackage{mathtools}
\usepackage{url}
\usepackage{xcolor}
	\definecolor{Blue}{HTML}{3d25b9}
	\definecolor{Pink}{HTML}{ec028d}

\usepackage{enumitem}
	\setlist{topsep=0pt,itemsep=0pt}
\usepackage[top=25mm,bottom=25mm,left=24.8mm,right=24.8mm]{geometry}
\usepackage{microtype}
\usepackage{sectsty}
	\sectionfont{\large}
	\subsectionfont{\normalsize}
\usepackage{titlesec}
	\titlespacing{\section}{0pt}{12pt}{0pt}
	\titlespacing{\subsection}{0pt}{6pt}{0pt}
\usepackage{titling}
\usepackage[nottoc]{tocbibind} 
\usepackage{tocloft}

\usepackage{thmtools}
\usepackage[colorlinks=true,linkcolor=Blue,citecolor=Blue]{hyperref}
	\hypersetup{breaklinks=true}
\usepackage[capitalise,noabbrev]{cleveref}
	\crefname{equation}{equation}{equations}
	\crefname{conjecture}{Conjecture}{Conjectures}

\theoremstyle{plain}
	\newtheorem{theorem}{Theorem}
	\newtheorem{proposition}[theorem]{Proposition}

\theoremstyle{definition}
	\newtheorem{definition}[theorem]{Definition}

\newcommand{\dd}{\mathrm{d}}

\renewcommand{\leq}{\leqslant}
\renewcommand{\geq}{\geqslant}

\newcommand{\Res}{\mathop{\,\rm Res\,}}

\newcommand{\T}{\mathcal{T}}
\newcommand{\modm}{\mathcal{M}}
\newcommand{\LL}{\bm{L}}
\newcommand{\bq}{\mathbb{Q}}
\newcommand{\br}{\mathbb{R}}
\newcommand{\bz}{\mathbb{Z}}

\newcommand{\cc}{\mathcal{C}}
\newcommand{\ce}{\mathcal{E}}
\newcommand{\co}{\mathcal{O}}
\newcommand{\cu}{\mathcal{U}}

\title{A $q$-analogue of Mirzakhani's recursion for Weil--Petersson volumes}
\author{Norman Do \and Paul Norbury}

\begin{document}

\makeatletter
\textbf{\large \thetitle}

\textbf{\theauthor}
\makeatother

School of Mathematics, Monash University, VIC 3800 Australia \\
School of Mathematics and Statistics, The University of Melbourne, VIC 3010 Australia \\
Email: \href{mailto:norm.do@monash.edu}{norm.do@monash.edu}, \href{mailto:norbury@unimelb.edu.au}{norbury@unimelb.edu.au}

{\em Abstract.} We define $q$-analogues of Mirzakhani's recursion for Weil--Petersson volumes and the Stanford--Witten recursion for super Weil--Petersson volumes. Okuyama recently introduced a $q$-deformation of the Gaussian Hermitian matrix model which produces quasi-polynomials that recover the Weil--Petersson volumes via a rescaled $q \to 1$ limit. The $q$-deformations of the Weil--Petersson volumes produced here agree with the top degree terms of Okuyama's quasi-polynomials and suggest a variation of Okuyama's methods to the super setting.

\emph{Acknowledgements.} The first author was supported by the Australian Research Council grant FT240100795.

\emph{2020 Mathematics Subject Classification.} 14H10, 14D23, 32G15

~

\hrule

~

\tableofcontents

~

\hrule

~

\section{Introduction} \label{sec:introduction}

The Weil--Petersson volume of the moduli space $\modm_{g,n}$ of genus $g$ curves with $n$ marked points is defined from the K\"ahler form $\omega^{\mathrm{WP}}$ of the Weil--Petersson metric by
\[
V^{\mathrm{WP}}_{g,n} = \int_{\modm_{g,n}} \!\! \exp \omega^{\mathrm{WP}}.
\]
The Weil--Petersson form $\omega^{\mathrm{WP}}$ can be defined purely in terms of the complete hyperbolic structure associated to a genus $g$ curve with $n$ marked points. Wolpert proved that the Fenchel--Nielsen coordinates on the Teichm\"uller space $\T_{g,n}$ are Darboux coordinates for $\omega^{\mathrm{WP}}$~\cite{WolWei}. More generally, one can use Fenchel--Nielsen coordinates on the Teichm\"uller space $\T_{g,n}(L_1, \ldots, L_n)$ of hyperbolic surfaces with geodesic boundary components of lengths $(L_1, \ldots, L_n) \in\br_{\geq 0}^n$. Wolpert's theorem then enables one to define a family of deformations $\omega^{\mathrm{WP}}(L_1, \ldots, L_n)$ of the Weil--Petersson form $\omega^{\mathrm{WP}}=\omega^{\mathrm{WP}}(0, \ldots, 0)$ with associated volumes 
\[
V^{\mathrm{WP}}_{g,n}(L_1, \ldots, L_n)=\int_{\modm_{g,n}} \!\! \exp\omega^{\mathrm{WP}}(L_1, \ldots, L_n).
\]

Mirzakhani produced relations -- see \cref{volrecWP} in \cref{sec2} -- between the volumes $V^{\mathrm{WP}}_{g,n}(L_1, \ldots, L_n)$ that determine the volumes recursively from $V^{\mathrm{WP}}_{0,3}(L_1,L_2,L_3)$ and $V^{\mathrm{WP}}_{1,1}(L_1)$,~\cite{MirSim}. She used this recursion to prove that the volumes are polynomial in the lengths and moreover that $V^{\mathrm{WP}}_{g,n}(L_1, \ldots, L_n) \in \bq[\pi^2][L_1^2, \ldots, L_n^2]$.

The main purpose of this paper is to introduce a natural $q$-deformation of Mirzakhani's recursion that produces a family of polynomials $V_{g,n}(L_1, \ldots, L_n) \in \bq[[q]][L_1^2, \ldots, L_n^2]$, whose coefficients are $q$-series. The recursion requires integration using the kernels $D_q(x,y,z)$ and $R_q(x,y,z)$, which we now introduce. These are defined from the function
\begin{equation} \label{Hq}
H_q(x,y) = \frac{1}{2} \sum_{m=1}^\infty(-1)^{m-1} q^{m^2/2} (q^{m/2} + q^{-m/2}) \Big(e^{\frac{1}{2}(x+y) (q^{m/2}-q^{-m/2})} + e^{\frac{1}{2}(x-y) (q^{m/2}-q^{-m/2})} \Big),
\end{equation}
via the formulas
\[
\frac{\partial}{\partial x} D_q(x,y,z) = H_q(y+z,x) \qquad \text{and} \qquad \frac{\partial}{\partial x} R_q(x,y,z) = \frac{1}{2} \big( H_q(z,x+y) + H_q(z,x-y) \big),
\]
together with the initial conditions $D_q(0,y,z) = R_q(0,y,z) = 0$.

Define the polynomials $V_{g,n}(L_1, \ldots, L_n) \in\bq[[q]][L_1^2, \ldots, L_n^2]$ from the base cases
\begin{align}
V_{0,1}(L) &:= 0 \notag \\
V_{0,2}(L_1, L_2) &:= 0 \notag \\
V_{0,3}(L_1,L_2,L_3) &:= 1 \notag \\
V_{1,1}(L) &:= \frac{1}{2L} \int_0^\infty \!\! xD_q(L,x,x) \, \dd x, \label{v11}
\end{align}
and the recursion
\begin{align} 
L_1&V_{g,n}(L_1, \LL_K)=\sum_{j=2}^n\int_0^\infty \!\! x R_q(L_1,L_j,x) \, V_{g,n-1}(x, \LL_{K \setminus \{j\}}) \, \dd x \notag \\
&+ \frac{1}{2} \int_0^\infty \!\! \int_0^\infty \!\! xy D_q(L_1, x, y) \Big[V_{g-1,n+1}(x, y, \LL_K) + \sum_{\substack{g_1+g_2=g \\ I \sqcup J = K}} V_{g_1,|I|+1}(x, \LL_I) \, V_{g_2,|J|+1}(y, \LL_J) \Big] \, \dd x \, \dd y. \label{qvolrec}
\end{align}
Here, we use the notation $K = \{2, 3, \ldots, n\}$ and write $\LL_I = (L_{i_1}, L_{i_2}, \ldots, L_{i_m})$ for $I = \{i_1, i_2, \ldots, i_m\}$.

Calculating the integral of \cref{v11} and applying the recursion of \cref{qvolrec} yields the formulas
\begin{align*}
V_{1,1}(L) &= \frac{1}{48} L^2 + \frac{1}{2} \zeta_q(2) \\
V_{0,4} (L_1, L_2, L_3, L_4) &= \frac{1}{2} \sum_{i=1}^4 L_i^2 + 12 \zeta_q(2) \\
V_{1,2}(L_1, L_2) &= \frac{1}{192} (L_1^2+L_2^2)^2 + \frac{1}{2} \zeta_q(2) (L_1^2+L_2^2) + 5\zeta_q(4) + 7\zeta_q(2)^2.
\end{align*}
These use the $q$-zeta function evaluated at the even integers, which we define via the formula
\begin{equation} \label{eq:qzeta}
\zeta_q(2k)=\sum_{m=1}^\infty \frac{q^{mk}}{(1-q^m)^{2k}}.
\end{equation}
Observe that this is a $q$-series that converges for $|q|<1$ and satisfies
\[
\lim_{q \to 1} (1-q)^{2k} \, \zeta_q(2k) = \zeta(2k).
\]

\begin{theorem} \label{thm:main}
The recursion of \cref{qvolrec} defines symmetric polynomials $V_{g,n}(L_1, \ldots, L_n) \in\bq[[q]][L_1^2, \ldots, L_n^2]$ which tend to the Weil--Petersson volumes in the following rescaled $q \to 1$ limit.
\[
\lim_{q \to 1} (1-q)^{6g-6+2n} \, V_{g,n} \Big( \frac{L_1}{1-q}, \ldots, \frac{L_n}{1-q} \Big) = V^{\mathrm{WP}}_{g,n}(L_1, \ldots, L_n)
\]
\end{theorem}

\cref{thm:main}, which is the first main result of the paper, comprises two parts. The first shows that the recursion produces symmetric polynomials with $q$-series coefficients. The second makes sense of the $q \to 1$ limit using the fact that the coefficients are convergent for $|q|<1$ and then calculates the limit via a comparison of the recursion of \cref{qvolrec} with Mirzakhani's recursion for Weil--Petersson volumes. \Cref{thm:main} justifies considering $V_{g,n}(L_1, \ldots, L_n)$ as a $q$-deformation of the Weil--Petersson volume $V^{\mathrm{WP}}_{g,n}(L_1, \ldots, L_n)$.

The main content of the first part of \cref{thm:main} is nicely demonstrated by the calculation of the integral given by \cref{v11} that defines $V_{1,1}(L)$. Similar integrals arise more generally in the calculation of $V_{g,n}(L_1, \ldots, L_n)$ using the recursion of \cref{qvolrec}. A priori, we observe that $V_{1,1}(L) \in\bq[[q, q^{-1}]] [[L^2]]$ that is, the integral could contain terms with arbitrarily large powers of $L$ and negative powers of $q$.
\begin{align*}
V_{1,1}(L) &= \frac{1}{2L} \int_0^\infty \!\! x D_q(L,x,x) \, \dd x \\
&=\sum_{m=1}^\infty (-1)^{m-1}q^{\binom{m}{2}}(1+q^m)\frac{\sinh\frac12L(q^{m/2}-q^{-m/2})}{L(q^{m/2}-q^{-m/2})}
 \int_0^\infty xe^{\frac12x(q^{m/2}-q^{-m/2})}dx\\
 &=\sum_{m=1}^\infty (-1)^{m-1}q^{\binom{m}{2}}(1+q^m)\frac{\sinh\frac12L(q^{m/2}-q^{-m/2})}{L(q^{m/2}-q^{-m/2})}
\frac{4}{(q^{m/2}-q^{-m/2})^2}\\
&=\sum_{k=0}^\infty\frac{L^{2k}}{2^{2k+1}(2k+1)!}\sum_{m=1}^\infty(-1)^{m-1}q^{\binom{m}{2}}(1+q^m)(q^{m/2}-q^{-m/2})^{2k-2}\\
&=\frac{1}{48}L^2+\frac{1}{2}\zeta_q(2)
\end{align*}
The second equality uses the series for $D_q$ and the dominated convergence theorem to interchange the sum and the integral---see the proof of \cref{F2k+1}---and the fourth equality uses absolute convergence to interchange the order of summation. The final equality requires non-trivial $q$-series identities -- see \cref{qident} in \cref{sec2} -- to produce the coefficients $\frac{1}{48}$ and $\frac{1}{2}\zeta_q(2)$, and also to prove that the coefficients of $L^{2k}$ vanish for all $k \geq 2$. Observe that this vanishing requires the cancellation of terms that include negative powers of $q$.

The appearance of $\zeta_q(2)$ in $V_{1,1}(L)$ is a special case of a more general phenomenon. As a consequence of \cref{F2k+1} below, the coefficients of $V_{g,n}(L_1, \ldots, L_n)$ are polynomials in values of the $q$-zeta function.
\[
V_{g,n}(L_1, \ldots, L_n) \in \bq[\zeta_q(2), \zeta_q(4), \ldots, \zeta_q(6g-6+2n)] [L_1^2, \ldots, L_n^2]
\]

The second main result of the paper is an analogue of \cref{thm:main} in the context of super Weil--Petersson volumes. The notion of super Weil--Petersson volumes was defined by Stanford and Witten in their work on JT gravity and its generalisations~\cite{SWiJTG}. They produced a super analogue of Mirzakhani's recursion for Weil--Petersson volumes which was later proved by the second author~\cite{NorEnu}.

We introduce a natural $q$-deformation of the Stanford--Witten recursion that produces a family of polynomials $\widehat{V}_{g,n}^{(m)}(L_1, \ldots, L_n) \in\bq[[q]][L_1^2, \ldots, L_n^2]$ for integers $g \geq 0$, $n \geq 1$ and $m \geq 0$, following previous work of the second author~\cite{NorSup}. In fact, we prove that for $m$ odd, we have $\widehat{V}_{g,n}^{(m)}(L_1, \ldots, L_n) = 0$, while for $m$ even, we have
\[
\widehat{V}_{g,n}^{(m)}(L_1, \ldots, L_n) \in \bq[\zeta^{\mathrm{odd}}_q(2), \zeta^{\mathrm{odd}}_q(4), \ldots, \zeta^{\mathrm{odd}}_q(2g-2+m)] [L_1^2, \ldots, L_n^2].
\]
Here, we use the odd $q$-zeta function evaluated at the even integers, which we define via the following sum over the odd positive integers.
\[
\zeta^{\mathrm{odd}}_q(2k)=\sum_{m \text{ odd}} \frac{q^{mk}}{(1-q^m)^{2k}}
\]
This is a $q$-series that converges for $|q|<1$.

It is convenient to assemble the polynomials $\widehat{V}_{g,n}^{(m)}(L_1, \ldots, L_n)$ into series via the equation
\begin{equation} \label{supvolpol}
\widehat{V}_{g,n}(L_1, \ldots, L_n; s) := \sum_{m=0}^\infty \frac{s^m}{m!} \, \widehat{V}_{g,n}^{(m)}(L_1, \ldots, L_n),
\end{equation}
where we will often omit the variable $s$ in the notation as it is implicit.

Our $q$-deformation of the Stanford--Witten recursion requires integration using the kernels $\widehat{D}_q(x,y,z)$ and $\widehat{R}_q(x,y,z)$, which we now introduce. These are defined from the sum over the odd positive integers
\begin{align}
\widehat{H}_q(x,y) = \frac{1}{8\psi(q)^{2}} \sum_{m \text{ odd}} (-1)^{(m-1)/2} &q^{(m^2-1)/4}(q^{m/2}+q^{-m/2}) \notag \\
&\Big( e^{\frac{1}{4}(x-y)(q^{m/2}-q^{-m/2})}-e^{\frac{1}{4}(x+y)(q^{m/2}-q^{-m/2})} \Big), \label{Hqsup}
\end{align}
via the formulas
\[
\widehat{D}_q(x,y,z)=\widehat{H}_q(y+z,x) \qquad \text{and} \qquad \widehat{R}_q(x,y,z) = \frac{1}{2} \big(\widehat{H}_q(z,x+y)+\widehat{H}_q(z,x-y) \big).
\]
The definition of $\widehat{H}_q(x,y)$ requires the following Ramanujan theta function, which converges for $|q|<1$.
\[
\psi(q) = \prod_{m=1}^\infty\frac{1-q^{2m}}{1-q^{2m-1}}=\frac{(1-q^2)(1-q^4)(1-q^6) \cdots}{(1-q)(1-q^3)(1-q^5) \cdots} = 1+q+q^3+q^6+q^{10} + \cdots
\]

Define $\widehat{V}_{g,n}(L_1, \ldots, L_n;s) \in\bq[[q]] [L_1^2, \ldots, L_n^2] [[s^2]]$ from the initial conditions $\widehat{V}^{(0)}_{0,1}(L) = 0$, $\widehat{V}^{(1)}_{0,1}(L) = 0$, $\widehat{V}^{(2)}_{0,1}(L) = 1$ and $\widehat{V}^{(0)}_{1,1}(L_1)=\frac{1}{8}$, and the same recursion as \cref{qvolrec} although with the new kernels $\widehat{D}_q(x,y,z)$ and $\widehat{R}_q(x,y,z)$.
\begin{align}
L_1 &\widehat{V}_{g,n}(L_1, \LL_k;s) = \sum_{j=2}^n \int_0^\infty \!\! x \widehat{R}_q(L_1,L_j,x) \, \widehat{V}_{g,n-1}(x, L_2, \LL_{K \setminus \{j\}};s) \, \dd x \notag \\
&+ \frac{1}{2} \int_0^\infty \!\! \int_0^\infty \!\! xy \widehat{D}_q(L_1,x,y) \, \Big[ \widehat{V}_{g-1,n+1}(x,y,L_K;s) +\hspace{-2mm} \mathop{\sum_{g_1+g_2=g}}_{I \sqcup J = K} \hspace{-2mm}\widehat{V}_{g_1,|I|+1}(x, \LL_I;s) \, \widehat{V}_{g_2,|J|+1}(y, \LL_J;s) \Big] \, \dd x \, \dd y \label{qvolrecsup}
\end{align}
For example, the recursion produces all terms of the disk series, the first of which are as follows.
\[
\widehat{V}_{0,1}(L; s) = \frac{s^2}{2!} + \bigg( \frac{1}{2}L^2 + 48\zeta^{\mathrm{odd}}_q(2) \bigg) \frac{s^4}{4!} + \bigg( \frac{3}{8} L^4 + 240\zeta^{\mathrm{odd}}_q(2) L^2 +5760\zeta^{\mathrm{odd}}_q(4) + 17280\zeta^{\mathrm{odd}}_q(2)^2 \bigg) \frac{s^6}{6!} + \cdots
\]

\begin{theorem} \label{super}
The recursion of \cref{qvolrecsup} defines symmetric polynomials $\widehat{V}_{g,n}^{(m)}(L_1, \ldots, L_n) \in \bq[[q]] [L_1^2, \ldots, L_n^2]$. The series defined by \cref{supvolpol} tend to corresponding series for super Weil--Petersson volumes in the following rescaled $q\to 1$ limit.
\[
\lim_{q \to 1} (1-q)^{2g-2} \, \widehat{V}_{g,n}\Big( \frac{L_1}{1-q}, \ldots, \frac{L_n}{1-q}; (1-q)s \Big) = \widehat{V}^{\mathrm{WP}}_{g,n}(L_1, \ldots, L_n; s)
\]
\end{theorem}

The polynomials $V_{g,n}(L_1, \ldots, L_n)$ introduced here arise naturally in recent work of Okuyama on the double-scaled SYK model, in which he defined quasi-polynomials $N^q_{g,n}(b_1, \ldots, b_n)$ that he interpreted as discrete analogues of the Weil--Petersson volumes~\cite{OkuDis}.  Okuyama constructed $N^q_{g,n}(b_1, \ldots, b_n)$ from a certain Hermitian matrix model and its associated loop equations.  It is proven in \cite{DNoDou} that $V_{g,n}(b_1, \ldots, b_n)$ defined via \eqref{qvolrec} agrees with the top degree terms of $N^q_{g,n}(b_1, \ldots, b_n)$ with respect to the grading $\deg b_i=1$ and $\deg\zeta_q(2k)=2k$.  Together with Theorem~\ref{thm:main}, this proves Okuyama's conjectural relationship of the double-scaled SYK model with Weil--Petersson volumes.

The geometry underlying the constructions in this paper provides an interesting topic that deserves further study. The functions $D_q(x,y,z)$ and $R_q(x,y,z)$ are related to the functions $D(x,y,z)$ and $R(x,y,z)$ defined by Mirzakhani via a rescaled $q \to 1$ limit -- see \cref{qkerlim}. Mirzakhani produced the kernels $D(x,y,z)$ and $R(x,y,z)$ from the geometry of a hyperbolic pair of pants with geodesic boundary components of lengths $x, y, z$~\cite{MirSim}. They arise as natural lengths along these geodesic boundaries and appear in Mirzakhani's generalisation of McShane's identity~\cite{McSRem}. Similarly, the rescaled $q \to 1$ limits of the functions $\widehat{D}_q(x,y,z)$ and $\widehat{R}_q(x,y,z)$ are the functions $\widehat{D}(x,y,z)$ and $\widehat{R}(x,y,z)$ appearing in the recursion for super Weil--Petersson volumes -- see \cref{qkerlimsup}. Stanford and Witten produced geometric constructions for $\widehat{D}(x,y,z)$ and $\widehat{R}(x,y,z)$ using supergeometry~\cite{SWiJTG}. Geometric interpretations of the $q$-analogues $D_q(x,y,z)$ and $R_q(x,y,z)$, as well as their super counterparts $\widehat{D}_q(x,y,z)$ and $\widehat{R}_q(x,y,z)$, would be extremely interesting.

In another direction, one might hope to interpret the polynomials we introduce here as ``$q$-volumes'' defined by integration of $q$-de Rham cohomology classes, following Aomoto~\cite{AomQan}. This may be related to the recently defined notion of the Habiro ring~\cite{GSWZHab}, which raises the question of whether the constructions in this paper exhibit interesting behaviour as $q$ approaches roots of unity.

\Cref{sec2} contains the proof of \cref{thm:main}, which is divided into \cref{prop:qident,F2k+1,limit}. In \cref{sec3}, we define the Weil--Petersson super volumes before giving the proof of \cref{super}, which is again divided into more basic pieces -- namely, \cref{prop:qidentodd,prop:F2k+1odd,limitsup}.

\section{A \texorpdfstring{$q$}{q}-deformation of Weil--Petersson volumes} \label{sec2}

In this section, we prove the existence and properties of the integrals in the recursion of \cref{qvolrec}.  The function $H_q(x,y)$ is treated as a formal power series in $y$, hence $D_q(x,y,z)$ and $R_q(x,y,z)$ are formal power series in $x$, respectively $x$ and $y$.  We prove convergence of integration of $H_q(x,y)$ with respect to $x$, which gives convergence of integration of $D_q(x,y,z)$, respectively $R_q(x,y,z)$, with respect to $y$ and $z$, respectively $z$.     The same holds for $\widehat{H}_q(x,y)$, $\widehat{D}_q(x,y,z)$ and $\widehat{R}_q(x,y,z)$.

We begin with a sequence of $q$-series identities.

\begin{proposition} \label{prop:qident}
For any $k\in\bz$,
\begin{equation} \label{qident}
\sum_{m=1}^\infty(-1)^{m-1}q^{m(m-1)/2}(1+q^m)\frac{q^{mk}}{(1-q^m)^{2k}} = s_k \big( \zeta_q(2), \zeta_q(4), \ldots \big),
\end{equation}
where $s_k$ is defined by
\begin{equation} \label{eq:sk}
\exp \bigg(\sum_{m=1}^\infty \frac{p_m}{m} \, t^m \bigg) = \sum_{k=-\infty}^\infty s_k(p_1, p_2, \ldots) \, t^k.
\end{equation}
For example, we have $s_k(p_1, p_2, \ldots) = 0$ for $k < 0$ and
\[
s_0(p_1, p_2, \ldots) = 1, \  s_1(p_1, p_2, \ldots) = p_1, \quad s_2(p_1, p_2, \ldots) = \frac{1}{2}(p_2+p_1^2), \quad  s_3(p_1, p_2, \ldots) = \frac{1}{6}(2p_3 + 3p_1p_2 + p_1^3).
\]
\end{proposition}

\begin{proof}
The case $k=0$ is equivalent to the following telescoping sum, which converges for $|q| < 1$.
\begin{equation} \label{eq:kequals0}
\sum_{m=1}^\infty(-1)^{m-1}q^{m(m-1)/2}(1+q^m) = (1+q) - (q+q^3) + (q^3+q^6) - (q^6 + q^{10}) + \cdots = 1
\end{equation}

For the case $k > 0$, we begin with the following infinite partial fraction decomposition, which appears in the work of Andrews~\cite[Equation~(2.1)]{AndHec}.
\[
\prod_{m=1}^\infty\frac{(1-q^m)^2}{(1-zq^m) (1-z^{-1}q^{m-1})}=\sum_{m=0}^\infty\frac{(-1)^mq^{m(m+1)/2}}{1-z^{-1}q^m}-z\sum_{m=1}^\infty\frac{(-1)^mq^{m(m+1)/2}}{1-zq^m}
\]
Multiply both sides by $1-z^{-1}$ to obtain
\begin{align*}
\prod_{m=1}^\infty\frac{(1-q^m)^2}{(1-zq^m)(1-z^{-1}q^m)}&=1+(1-z^{-1})\sum_{m=1}^\infty(-1)^mq^{m(m+1)/2} \left( \frac{1}{1-z^{-1}q^m}-\frac{z}{1-zq^m} \right) \\
&=1+\sum_{m=1}^\infty(-1)^mq^{m(m+1)/2}\frac{(1-z^{-1})(1-z)(1+q^m)}{(1-z^{-1}q^m)(1-zq^m)}\\
&=1+\sum_{m=1}^\infty(-1)^mq^{m(m-1)/2}(1+q^m)\left(1-\frac{(1-q^m)^2}{(1-z^{-1}q^m)(1-zq^m)}\right)\\
&=\sum_{m=1}^\infty(-1)^{m-1}q^{m(m-1)/2}(1+q^m)\frac{(1-q^m)^2}{(1-z^{-1}q^m)(1-zq^m)},
\end{align*}
where the final equality uses \cref{eq:kequals0}. Observe that this expression converges for $1 \leq |z| < |q|^{-1}$. Now change coordinates to $t = i(z^{1/2}-z^{-1/2})$ so that $t^2=2-z-z^{-1}$ to obtain the identity
\begin{equation} \label{qid}
\prod_{m=1}^\infty\frac{(1-q^m)^2}{(1-q^m)^2+t^2q^m} = \sum_{m=1}^\infty(-1)^{m-1}q^{m(m-1)/2}(1+q^m)\frac{(1-q^m)^2}{(1-q^m)^2+t^2q^m}.
\end{equation}
This expression now converges for $t$ away from the poles and in particular for all $|t| < 1-q$.

Equating the coefficients of $t^{2k}$ in the Taylor expansions at $t=0$ of both sides of \cref{qid} produces the desired identity, as we will now see. For the left side of \cref{qid}, consider the Taylor series at $t = 0$ of its logarithm.
\begin{align*}
\log\prod_{m=1}^\infty\frac{(1-q^m)^2}{(1-q^m)^2+t^2q^m} &= -\sum_{m=1}^\infty \log \bigg( 1+\frac{t^2q^m}{(1-q^m)^2} \bigg) \\
&= \sum_{k=1}^\infty (-1)^k\frac{t^{2k}}{k}\sum_{m=1}^\infty \frac{q^{mk}}{(1-q^m)^{2k}} =\sum_{k=1}^\infty (-1)^k\frac{t^{2k}}{k}\zeta_q(2k)
\end{align*}
Now take the exponential of both sides and use \cref{eq:sk} to obtain
\begin{equation} \label{eq:prooflhs}
\prod_{m=1}^\infty\frac{(1-q^m)^2}{(1-q^m)^2+t^2q^m}=\sum_{k=0}^\infty (-1)^k \, s_k(\zeta_q(2), \zeta_q(4), \ldots) \, t^{2k}.
\end{equation}
For the right side of \cref{qid}, the Taylor series at $t=0$ can be calculated termwise, as the factor $q^{m(m-1)/2}$ guarantees the convergence of the partial sums. The result is given by
\begin{equation} \label{eq:proofrhs}
\sum_{k=0}^\infty (-1)^k\sum_{m=1}^\infty (-1)^{m-1}q^{m(m-1)/2}(1+q^m)\frac{q^{mk}}{(1-q^m)^{2k}} \, t^{2k}.
\end{equation}

Equating the $t^{2k}$ coefficients of the Taylor series in \cref{eq:prooflhs,eq:proofrhs} yields the desired result for $k > 0$.

For the case $k < 0$ put $\ell=-k$, denote by $F(t)$ the function given by the right hand side of \cref{qid} and integrate $t^{2\ell-1} F(t)$ over large circles. Using
\begin{equation} \label{contint}
 \frac{1}{2\pi i}\oint_{|t|=R} \frac{t^{2\ell-1}}{1+at^2} \, \dd t = \begin{cases}
(-1)^{\ell -1}a^{-\ell }, &|a|^{-1}<R^2, \\
0, &|a|^{-1}>R^2,
\end{cases}
\end{equation} 
leads to
\begin{equation} \label{largecirc} 
 \frac{(-1)^{\ell -1}}{2\pi i}\oint_{|t|=R} \hspace{-5mm}t^{2\ell -1} \, \dd t \sum_{m=1}^\infty(-1)^{m-1}q^{\binom{m}{2}}(1+q^m)\frac{(1-q^m)^2}{(1-q^m)^2+t^2q^m} 
= \hspace{-1mm}\sum_{m=1}^N(-1)^{m-1}q^{\binom{m}{2}}(1+q^m)\frac{(1-q^m)^{2\ell }}{q^{m\ell }}
\end{equation}
Here, $N$ is determined by $q^{-m/2}-q^{m/2}<R \Leftrightarrow m \leq N$ and is necessarily finite. The dominated convergence theorem allows us to interchange the integral and the sum. 


The poles of the integrand of the left hand side of \cref{largecirc} occur at $\pm ia_m$ for
\[a_m:=q^{-m/2}-q^{m/2},\] 
which is a geometrically increasing sequence, due to $a_m/a_{m+1}<q^{1/2}$.  
This allows us to choose the contour in \cref{largecirc} so that its distance from the poles is bounded below, leading to an upper bound on the integrand as follows.   Define 
\[R_M^2=a_Ma_{M+1}\]
so that $a_m<R_M\Leftrightarrow m\leq M$.    
For $|t|=R_M$ and $m\leq M$,
\[
|a_m^2+t^2|
\geq \bigl||t|^2-a_m^2\bigr|
=R_M^2-a_m^2>R_M^2-a_M^2>(1-q^{1/2})R_M^2.
\]
From
\[
q^{-m/2}>a_m\geq (1-q)q^{-m/2}
\]
we have
\[
\prod_{m=1}^M a_m^2
< q^{-M(M+1)/2},
\]
and
\[
R_M^{2M}
=(a_Ma_{M+1})^M
>
(1-q)^{2M}q^{-(2M+1)M/2}.
\]
Thus
\[
\prod_{m=1}^M
\left|
\frac{a_m^2}{a_m^2+t^2}
\right|
<\frac{\prod_{m=1}^M a_m^2}{(1-q^{1/2})^{M}R_M^{2M}}
<\frac{q^{(2M+1)M/2}q^{-M(M+1)/2}}{(1-q^{1/2})^{M}(1-q)^{2M}}=C_q^{M} q^{M^2/2}
\]
for $C_q=(1-q^{1/2})^{-1}(1-q)^{-2}$.

For the remaining factors, $m>M$, we have $R^2_M/a_m^2<q^{m-M-1/2}$, hence on $|t|=R_M$, 
\[
\prod_{m>M}
\left|
\frac{a_m^2}{a_m^2+t^2}
\right|
\leq
\prod_{m>M}\frac{1}{1-R_M^2/a_m^2}
<\prod_{j=1}^\infty\frac{1}{1-q^{j-1/2}}=\frac{1}{(q^{1/2};q)_\infty}<\infty
\]
for $(q^{1/2};q)_\infty\neq0$ the $q$-Pochhammer symbol. 
Now $R_M<q^{-(M+1)/2}$, so
\[
\left|
\frac{1}{2\pi}\oint_{|t|=R_M}
\hspace{-4mm}t^{2\ell -1}\,dt
\prod_{m=1}^\infty
\frac{a_m^2}{a_m^2+t^2}
\right|
<\frac{1}{(q^{1/2};q)_\infty}q^{-\ell (M+1)}C_q^M q^{M^2/2}\underset{M\to\infty}{\longrightarrow}0
\]
since $q^{M^2/2}$ decays faster than any exponential in $M$.  
Hence, for all $\ell >0$,
\[
\lim_{N\to\infty} \sum_{m=1}^N (-1)^{m-1} q^{m(m-1)/2} (1+q^m) \frac{(1-q^m)^{2\ell }}{q^{m\ell }} = 0,
\]
which corresponds precisely to the desired result in the case $k=-\ell<0$.
\end{proof}

From the function $H_q(x,y)$ defined in \cref{Hq}, we define
\[
F_{2k+1}(y)=\int_0^\infty \!\! x^{2k+1} H_q(x,y) \, \dd x.
\]
Here, $H_q(x,y)$ is to be understood as a series in $y^2$ -- more precisely,
\[
H_q(x,y) = \cosh \Big(y\frac{\dd}{\dd x} \Big) h_q(x) = \sum_{n=0}^\infty \frac{y^{2n}}{(2n)!} \frac{\dd^{2n}}{\dd x^{2n}} h_q(x),
\]
where 
\[
h_q(x) = \sum_{m=1}^\infty(-1)^{m-1}q^{m(m-1)/2}(1+q^m) \, e^{\frac{1}{2} x (q^{m/2}-q^{-m/2})}.
\]
Integrability of $H_q(x,y)$ in $x$ is guaranteed by the following result.

\begin{proposition} \label{F2k+1}
For each non-negative integer $k$, we have
\[
\frac{F_{2k+1}(y)}{(2k+1)!} = \int_0^\infty \!\! \frac{x^{2k+1}}{(2k+1)!} \, H_q(x,y) \, \dd x = \sum_{n=0}^{k+1} b_n \frac{y^{2k+2-2n}}{(2k+2-2n)!},
\]
where $b_0, b_1, b_2, \ldots \in \bq[\zeta_q(2), \zeta_q(4), \ldots]$ are defined by
\[
\sum_{n=0}^\infty b_nz^{2n} = \exp \bigg(\sum_{j=1}^\infty\frac{\zeta_q(2j)}{j} (4z^2)^j \bigg) = 1+4\zeta_q(2)z^2+8(\zeta_q(2)^2+\zeta_q(4))z^4 + \cdots.
\]
\end{proposition}

\begin{proof}
Integrability of $x^{2k+1}H_q(x,y)$ is a consequence of integrability of $x^{2k+1} \frac{\dd^{2n}}{\dd x^{2n}} h_q(x)$ for all $n \geq 0$, together with vanishing of the integrals for $n>k+1$.

We first prove integrability of $x^{2k+1} \frac{\dd^{2n}}{\dd x^{2n}} h_q(x)$ in the $n=0$ case. Consider the following inequality, which is satisfied for $0<q<1$ and any positive integer $m$.
\begin{equation} \label{qmcomp} 
q^{-m/2}-q^{m/2}=q^{-m/2}(1+q+q^2+ \cdots +q^{m-1})(1-q)> (q^{-m/2}+m-1)(1-q)>m(1-q)
\end{equation}
Here, the first inequality uses $q^{-m/2+k}+q^{m/2-k}>2$ for $k=1,2, \ldots, m-1$. A consequence of \cref{qmcomp} is the bound
\[
|h_q(x)| < \sum_{m=1}^\infty q^{m(m-1)/2}(1+q^m) \, e^{\frac{1}{2} x (q^{m/2}-q^{-m/2})} < 2\sum_{m=1}^\infty e^{-\frac{1}{2}xm(1-q)} = \frac{2}{e^{\frac{1}{2}(1-q)x}-1}.
\]
We know that $\frac{2x^{2k+1}}{e^{\frac{1}{2}(1-q)x}-1}$ is integrable for each $k\geq 0$. Furthermore, \cref{qmcomp} allows us to deduce pointwise convergence of the following partial sums, via upper bounds on the tail of the series.
\begin{multline*}
\lim_{N\to\infty} \sum_{m=1}^N (-1)^{m-1}q^{m(m-1)/2}(1+q^m) \, e^{\frac{x}{2(1-q)} (q^{m/2}-q^{-m/2})} \\
= \sum_{m=1}^\infty (-1)^{m-1}q^{m(m-1)/2}(1+q^m) \, e^{\frac{x}{2(1-q)} (q^{m/2}-q^{-m/2})}
\end{multline*}
By the dominated convergence theorem, $x^{2k+1}h_q(x)$ is Lebesgue integrable in $x$, and we can interchange the sum and integral to obtain the following for $0<q<1$.
\begin{align*} 
\int_0^\infty \!\! \frac{x^{2k+1}}{(2k+1)!} \, h_q(x) \, \dd x &= \sum_{m=1}^\infty (-1)^{m-1}q^{m(m-1)/2}(1+q^m) \int_0^\infty \!\! \frac{x^{2k+1}}{(2k+1)!} \, e^{\frac{1}{2}x (q^{m/2}-q^{-m/2})} \, \dd x \\
&=\sum_{m=1}^\infty (-1)^{m-1}q^{m(m-1)/2}(1+q^m)\frac{ 2^{2k+2}}{(q^{m/2}-q^{-m/2})^{2k+2}} \\
&=2^{2k+2} \, s_{k+1}\left(\zeta_q(2), \zeta_q(4), \ldots \right)
\end{align*} 
The last equality uses the $q$-series identity of \cref{qident}, where $s_k$ is the polynomial defined by \cref{eq:sk}, which leads to the definition of $b_n$ in the statement of the proposition.

When $n>0$, the proof of integrability of $x^{2k+1} \frac{\dd^{2n}}{\dd x^{2n}} h_q(x)$ is similar to the $n=0$ case.  Partition $h_q(x)=h_q^{(2n)}(x)+$ tail into the sum of the first $2n$ terms and the tail. Integrability of the first $2n$ terms of $x^{2k+1} \frac{\dd^{2n}}{\dd x^{2n}} h_q(x)$ again uses the inequality given in \cref{qmcomp}:
\begin{align*}
\frac{\dd^{2n}}{\dd x^{2n}} h_q^{(2n)}(x) 
&= \sum_{m=1}^{2n} q^{m(m-1)/2}(1+q^m)2^{-2n} \left(q^{m/2}-q^{-m/2}\right)^{2n}e^{\frac{x}{2} (q^{m/2}-q^{-m/2})} \\
&< 2^{1-2n}(1-q)^{2n}\sum_{m=1}^{2n} m^{2n}e^{-\frac{1}{2}xm(1-q)}
\end{align*}
The tail is dominated by an integrable function as follows.
\[
\sum_{m=2n+1}^\infty q^{m(m-1)/2}(1+q^m) (q^{m/2}-q^{-m/2})^{2n}e^{\frac{x}{2}\left(q^{m/2}-q^{-m/2}\right)} < 2\sum_{m=1}^\infty e^{-\frac{1}{2}xm(1-q)}=\frac{2}{e^{\frac{1}{2}(1-q)x}-1}
\]
The condition $m>2n$ ensures that $m(m-1)/2-mn\geq0$ so that $q^{m(m-1)/2}(1+q^m)\left(q^{m/2}-q^{-m/2}\right)^{2n}<2$. Hence, $\left| x^{2k+1} \frac{\dd^{2n}}{\dd x^{2n}} h_q(x) \right|$ is bounded above by an integrable function.
The same comparison on the tail proves pointwise convergence for the series. So as in the $n=0$ case, by the dominated convergence theorem, $x^{2k+1} \frac{\dd^{2n}}{\dd x^{2n}} h_q(x)$ is Lebesgue integrable in $x$ and we can interchange the sum and integral. The integral can be obtained from the $n=0$ case via integration by parts or it can be calculated as follows.
\begin{align*} 
& \int_0^\infty \!\! \frac{x^{2k+1}}{(2k+1)!} \frac{\dd^{2n}}{\dd x^{2n}} h_q(x) \, \dd x \\
={}& \sum_{m=1}^\infty (-1)^{m-1}q^{m(m-1)/2}(1+q^m)\tfrac{1}{2^{2n}} (q^{m/2}-q^{-m/2})^{2n} \int_0^\infty \!\! \frac{x^{2k+1}}{(2k+1)!}e^{\frac{x}{2}(q^{m/2}-q^{-m/2})} \, \dd x \\
={}& \sum_{m=1}^\infty (-1)^{m-1}q^{m(m-1)/2}(1+q^m)\frac{ 2^{2k+2-2n}}{(q^{m/2}-q^{-m/2})^{2k+2-2n}} \\
={}& 2^{2k+2-2n} \, s_{k+1-n}\left(\zeta_q(2), \zeta_q(4), \ldots \right).
\end{align*}

When $n>k+1$, the $q$-series identity of \cref{qident} implies the vanishing
\[
\sum_{m=1}^\infty (-1)^{m-1}q^{m(m-1)/2}(1+q^m)(q^{m/2}-q^{-m/2})^{2n-2k-2} = 0.
\]

Putting this together, we obtain
\begin{align*}
\int_0^\infty \!\! \frac{x^{2k+1}}{(2k+1)!} \, H(x,y) \, \dd x 
=\sum_{i=0}^{k+1} s_i ( \zeta_q(2), \zeta_q(4), \ldots) \, \frac{(2y)^{2k+2-2i}}{(2k+2-2i)!}.
\end{align*}
Equivalently, $F_{2k+1}(y)$ is a polynomial in $y$ given by
\[
\frac{F_{2k+1}(y)}{(2k+1)!}=\sum_{n=0}^{k+1} b_n\frac{y^{2k+2-2n}}{(2k+2-2n)!} 
\]
where the factor of $2^{2k+2-2n}$ is absorbed by the definition of $b_n$.
\end{proof}

\begin{proposition} \label{sym}
The quantity $V_{g,n}(L_1, \ldots, L_n)$ produced by the recursion of \cref{qvolrec} is a symmetric polynomial in $L_1^2, \ldots, L_n^2$.
\end{proposition}

\begin{proof}
The kernel $D_q(x,y,z)$ in the recursion of \cref{qvolrec} can be written explicitly as
\[
D_q(x,y,z) = \sum_{m=1}^\infty (-1)^{m-1} \frac{q^{m(m-1)/2}(1+q^m)}{q^{m/2}-q^{-m/2}}\Big( e^{\frac{x+y+z}{2} [m]_q} - e^{\frac{-x+y+z}{2} [m]_q} \Big),
\]
where we interpret it as a series in $x$ and introduce the shorthand $[m]_q = q^{m/2}-q^{-m/2}$. Similarly, the kernel $R_q(x,y,z)$ can be written explicitly as
\[
R_q(x,y,z)=\frac{1}{2}\sum_{m=1}^\infty\frac{(-1)^{m-1}q^{m(m-1)/2}(1+q^m)}{q^{m/2}-q^{-m/2}}\Big(e^{\frac{x+y+z}{2} [m]_q}+e^{\frac{x-y+z}{2} [m]_q}-e^{\frac{-x+y+z}{2} [m]_q}-e^{\frac{-x-y+z}{2} [m]_q}\Big),
\]
where we interpret it as a series in $x$ and $y$.

Integrability of $z^{2k+1}R_q(x,y,z)$ in $z$ is equivalent to integrability of $z^{2k+1}\frac{\partial}{\partial x}R_q(x,y,z)$ since $R_q(x,y,z)$ is a formal power series in $x$ and $y$ and $R_q(0,y,z)=0$, i.e. it is the {\em coefficients} of the series in $x$ and $y$ that we require to be integrable in $z$, and their integrability is unaffected by $\frac{\partial}{\partial x}$. Integrability of $z^{2k+1}\frac{\partial}{\partial x}R_q(x,y,z)$ is an immediate consequence of integrability of $x^{2k+1}H_q(x,y)$. Furthermore,
\[
\int_0^\infty \!\! z^{2k+1}\frac{\partial}{\partial x}R_q(x,y,z) \, \dd z=\frac{1}{2}(F_{2k+1}(x+y)+F_{2k+1}(x-y)),
\]
from which we can immediately retrieve $\int_0^\infty \! z^{2k+1}R_q(x,y,z) \, \dd z$ since the polynomial in $x$ stores coefficients and the linear operator $\frac{\partial}{\partial x}$ does not actually differentiate under the integral sign.

Integrability of $y^{2i+1}z^{2j+1}D_q(x,y,z)$ in $y$ and $z$ is equivalent to integrability of $y^{2i+1}z^{2j+1}\frac{\partial}{\partial x}D_q(x,y,z)$, again since $D_q(0,y,z)=0$ and it is the coefficients of the series in $x$ that we require to be integrable. The derivative $\frac{\partial}{\partial x}$
cancels the $q^{m/2}-q^{-m/2}$ denominator, leaving the same comparison integrable function as used for $H_q(x,y)$ to apply the dominated convergence theorem and deduce integrability. Furthermore, Fubini's theorem allows one to use the change of coordinates $y=\frac{1}{2}(u+v)$, $z=\frac{1}{2}(u-v)$ to prove
\begin{align*}
\int_0^\infty \!\! \int_0^\infty \!\! y^{2j+1} z^{2j+1} \frac{\partial}{\partial x} D_q(x,y,z) \, \dd y \, \dd z &= \int_0^\infty \!\! \int_0^\infty \!\! y^{2i+1} z^{2j+1} H(y+z,x) \, \dd y \, \dd z \\
&= \frac{(2i+1)! \, (2j+1)!}{(2i+2j+3)!} \, F_{2i+2j+3}(x).
\end{align*}
Again, we can immediately retrieve $\int_0^\infty \! \int_0^\infty \! y^{2j+1} z^{2k+1} D_q(x,y,z) \, \dd y \, \dd z$ since $\frac{\partial}{\partial x}$ does not differentiate under the integral sign.

We have proven that the recursion of \cref{qvolrec} is integrable and defines a polynomial $L_1V_{g,n}(L_1, \ldots, L_n)$, by the inductive assumption that all $V_{g',n'}$ generated earlier by the recursion are polynomial. Since we have $D_q(0,y,z)=R_q(0,y,z) = 0$, the polynomial is divisible by $L_1$, proving that $V_{g,n}(L_1, \ldots, L_n)$ is also polynomial. Moreover, $V_{g,n}(L_1, \ldots, L_n)$ is polynomial in $L_1^2$, since the recursion \eqref{qvolrec} can be used to produce the polynomial $(\partial/\partial L_1)(L_1V_{g,n}(L_1, \ldots, L_n))$ as a linear combination of $F_1(L_1), F_3(L_1), \ldots$, which are inherently polynomial in $L_1^2$.   By induction, $V_{g,n}(L_1, \ldots, L_n)$ is polynomial in $L_j^2$ for $j>1$.

To prove that $V_{g,n}(L_1, \ldots, L_n)$ is symmetric, we use the following fact: for any polynomial $P(L_1, \ldots, L_n)$, if $P(L_1,L_2, \ldots, L_n)+P(L_2,L_1, \ldots, L_n)=(L_1+L_2) \, Q(L_1, \ldots, L_n)$ for a polynomial $Q$, then $Q$ is symmetric in $L_1$ and $L_2$. Furthermore, it is enough to prove that $P(L_1,-L_1,L_3, \ldots, L_n)+P(-L_1,L_1,L_3, \ldots, L_n)=0$.

Apply the recursion of \cref{qvolrec} with each of $L_1$ and $L_2$ as the distinguished argument, and add to obtain the following polynomial symmetric in $L_1$ and $L_2$.
\begin{align} \label{qvolrec12}
L_1V_{g,n}(L_1,L_2, \ldots, L_n)+L_2V_{g,n}(L_2,L_1, \ldots, L_n)=\int_0^\infty \!\! (\,\cdot\,) \, \dd x +\int_0^\infty \!\! \int_0^\infty \!\! (\,\cdot\,) \, \dd x \, \dd y
\end{align}
Now set $L_2=-L_1$ in the right side of \cref{qvolrec12}. For $j\neq1,2$ in the first term on the right side of \cref{qvolrec}, the integrands cancel
\[
R_q(L_1,L_j,x)V_{g,n-1}(x,-L_1, \LL_{K \setminus \{2, j\}})+R_q(-L_1,L_j,x)V_{g,n-1}(x,L_1, \LL_{K \setminus \{2,j\}})=0,
\]
since $V_{g,n-1}(x,-L_1, \ldots)=V_{g,n-1}(x,L_1, \ldots)$ and $R_q(-L_1,L_j,x)=-R_q(L_1,L_j,x)$.

For $j=1$ or $2$, the integrands also cancel
\[
R_q(L_1,L_j,x)V_{g,n-1}(x,L_3, \ldots, L_n)+R_q(-L_1,L_j,x)V_{g,n-1}(x,L_3, \ldots, L_n)=0,
\]
since $R_q(-L_1,L_j,x)=-R_q(L_1,L_j,x)$.

Similarly, in the remaining terms on the right side of \cref{qvolrec}, the integrands cancel under $L_2=-L_1$, since 
\[
D_q(L_1,x,y)P(L_2^2)+D_q(L_2,x,y)P(L_1^2) \stackrel{L_2=-L_1}{\longrightarrow} \big( D_q(L_1,x,y)+D_q(-L_1,x,y) \big) P(L_1^2)=0,
\]
where $P(L_2^2)$ is a polynomial obtained from simpler volumes and $D_q(-x,y,z)=-D_q(x,y,z)$. Thus, the expression in \cref{qvolrec12} vanishes and we conclude that $V_{g,n}(L_1,L_2, \ldots, L_n)$ is symmetric in $L_1$ and $L_2$. By induction, $V_{g,n}(L_1,L_2, \ldots, L_n)$ is symmetric in $L_2, \ldots, L_n$ and hence, it is symmetric in $L_1,L_2, \ldots, L_n$, as claimed.
\end{proof}

To prove that the rescaled $q \to 1$ limit of $V_{g,n}(L_1, \ldots, L_n)$ gives the Weil--Petersson volumes, we show that the recursion of \cref{qvolrec} tends to the following recursion of Mirzakhani for Weil--Petersson volumes.

\begin{theorem}[Mirzakhani~\cite{MirSim}] \label{th:mirzvolrec}
The Weil--Petersson volumes satisfy the following recursion for $2g-2+n \geq 2$.
\begin{align}
L_1&V^{\mathrm{WP}}_{g,n}(L_1, \LL_K) = \sum_{j=2}^n\int_0^\infty \!\! xR(L_1,L_j,x) \, V^{\mathrm{WP}}_{g,n-1}(x, \LL_{K \setminus \{j\}}) \, \dd x \notag \\
&+\frac{1}{2}\int_0^\infty \!\! \int_0^\infty \!\! xyD(L_1,x,y) \Big[ V^{\mathrm{WP}}_{g-1,n+1}(x,y, \LL_K) + \mathop{\sum_{g_1+g_2=g}}_{I \sqcup J = K} V^{\mathrm{WP}}_{g_1,|I|+1}(x, \LL_I) \, V^{\mathrm{WP}}_{g_2,|J|+1}(y, \LL_J) \Big] \, \dd x \, \dd y \label{volrecWP}
\end{align}
Here, we use the notation $K = \{2, 3, \ldots, n\}$ and write $\LL_I = (L_{i_1}, L_{i_2}, \ldots, L_{i_m})$ for $I = \{i_1, i_2, \ldots, i_m\}$.
\end{theorem}

The kernels in \cref{volrecWP} are defined by
\[
\frac{\partial}{\partial x}D(x,y,z)=H(x,y+z) \qquad \text{and} \qquad \frac{\partial}{\partial x}R(x,y,z) = \frac{1}{2} \big( H(z,x+y)+H(z,x-y) \big)
\]
and the initial conditions $D(0,y,z)=R(0,y,z)=0$, from the function
\[
H(x,y)=\frac{1}{1+e^{\frac{x+y}{2}}}+\frac{1}{1+e^{\frac{x-y}{2}}}.
\]
More explicitly, we have
\[
D(x,y,z) = 2\log \bigg( \frac{e^{\frac{x}{2}}+e^{\frac{y+z}{2}}}{e^{-\frac{x}{2}}+e^{\frac{y+z}{2}}} \bigg) \qquad \text{and} \qquad R(x,y,z) = x - \log \bigg( \frac{\cosh\frac{y}{2}+\cosh\frac{x+z}{2}}{\cosh\frac{y}{2}+\cosh\frac{x-z}{2}} \bigg).
\]
The following limits relate Mirzakhani's kernels to the $q$-kernels introduced in \cref{sec:introduction}.
\begin{align} \label{qkerlim}
\lim_{q\to 1} H_q \Big( \frac{x}{1-q}, \frac{y}{1-q} \Big) &= H(x,y) \notag \\
\lim_{q\to 1} (1-q) \, D_q \Big( \frac{x}{1-q}, \frac{y}{1-q}, \frac{z}{1-q} \Big) &= D(x,y,z) \\
\lim_{q\to 1} (1-q) \, R_q \Big( \frac{x}{1-q}, \frac{y}{1-q}, \frac{z}{1-q} \Big) &= R(x,y,z) \notag
\end{align}

The recursion of \cref{volrecWP} uniquely determines $V_{g,n}^{\mathrm{WP}}(L_1, \ldots, L_n)$ from the base cases
\begin{align}
V^{\mathrm{WP}}_{0,1}(L) &:= 0 \notag \\
V^{\mathrm{WP}}_{0,2}(L_1, L_2) &:= 0 \notag \\
V^{\mathrm{WP}}_{0,3}(L_1,L_2,L_3) &:= 1 \notag \\
V^{\mathrm{WP}}_{1,1}(L) &:= = \frac{1}{2L} \int_0^\infty \!\! xD(L,x,x) \, \dd x. \label{eq:v11wp}
\end{align}
Calculating the integral of \cref{eq:v11wp} and applying the recursion of \cref{volrecWP} yields the formulas
\begin{align*}
V^{\mathrm{WP}}_{1,1}(L) &= \frac{1}{48}L^2+ \frac{\pi^2}{12} \\
V^{\mathrm{WP}}_{0,4} (L_1, L_2, L_3, L_4) &= \frac{1}{2} \sum_{i=1}^4 L_i^2 + 2\pi^2 \\
V^{\mathrm{WP}}_{1,2}(L_1, L_2) &= \frac{1}{192} (L_1^2+L_2^2)^2 + \frac{\pi^2}{12} (L_1^2+L_2^2) + \frac{\pi^4}{4}.
\end{align*}

\begin{proposition} \label{limit}
The polynomials $V_{g,n}(L_1, \ldots, L_n) \in\bq[[q]][L_1^2, \ldots, L_n^2]$ tend to the Weil--Petersson volumes in the following rescaled $q \to 1$ limit.
\[
\lim_{q\to 1} (1-q)^{6g-6+2n} \, V_{g,n} \Big( \frac{L_1}{1-q}, \ldots, \frac{L_n}{1-q} \Big) = V^{\mathrm{WP}}_{g,n}(L_1, \ldots, L_n)
\]
\end{proposition}

\begin{proof}
The limit clearly holds in the case $(g,n) = (0,3)$ since $V_{0,3}(L_1,L_2,L_3) = 1$ and $V_{0,3}^{\mathrm{WP}}(L_1,L_2,L_3) = 1$. In the case $(g,n) = (1,1)$, we can calculate explicitly as follows.
\[
\lim_{q \to 1} (1-q)^2 \, V_{1,1} \Big( \frac{L_1}{1-q} \Big) = \lim_{q \to 1} (1-q)^2 \, \Big(\frac{1}{48}\Big(\frac{L_1}{1-q}\Big)^2+\frac{1}{2}\zeta_q(2) \Big) = \frac{1}{48}L_1^2+\frac{1}{2}\zeta(2)=V_{1,1}^{\mathrm{WP}}(L_1)
\]

Using the definition $F^M_{2k+1}(y)=\int_0^\infty \! x^{2k+1} H(x,y) \, \dd x$, Mirzakhani~\cite{MirSim} proved that
\[
\frac{F^M_{2k+1}(y)}{(2k+1)!}=\sum_{n=0}^{k+1}b^M_n \frac{y^{2k+2-2n}}{(2k+2-2n)!}, \qquad \text{where } \sum_{n=0}^\infty b^M_n z^{2n} = \frac{2\pi z}{\sin(2\pi z)}.
\]
By \cref{F2k+1}, we have
\[
\frac{F_{2k+1}(y)}{(2k+1)!}=\sum_{n=0}^{k+1} b_n\frac{y^{2k+2-2n}}{(2k+2-2n)!},\qquad \text{where } \sum_{n=0}^\infty b_n z^{2n} = \exp \bigg( \sum_{j=1}^\infty\frac{\zeta_q(2j)}{j} (4z^2)^j \bigg).
\]
The Weierstrass product
\[
\frac{\sin(2\pi z)}{2\pi z}=\prod_{n=1}^\infty \bigg( 1-\frac{4z^2}{n^2} \bigg)
\]
together with the expansion of $\log(1-t)$ gives the Taylor expansion
\[
\log\left(\frac{\sin(2\pi z)}{2\pi z}\right) = -\sum_{j=1}^\infty \frac{\zeta(2j)}{j} (4z^2)^{j} \qquad \Rightarrow \qquad \frac{2\pi z}{\sin(2\pi z)} = \exp \bigg(\sum_{j=1}^\infty \frac{\zeta(2j)(4z^2)^{j}}{j} \bigg).
\]
This proves that
\[
\lim_{q\to 1}(1-q)^{2k+2} \, F_{2k+1} \Big( \frac{y}{1-q} \Big) = F^M_{2k+1}(y).
\]

The polynomials $F^M_{2k+1}(y)$ uniquely determine the Weil--Petersson volumes $V^{\mathrm{WP}}_{g,n}(L_1, \ldots, L_n)$ from the recursion of \cref{volrecWP}, as shown by Mirzakhani~\cite{MirSim}. Similarly, the polynomials $F_{2k+1}(y)$ uniquely determine the polynomials $V_{g,n}(L_1, \ldots, L_n)$ from the recursion of \cref{qvolrec}, as shown here in the proof of \cref{sym}. We have shown that the rescaled limits of the initial polynomials $V_{0,3}$ and $V_{1,1}$ produce the Weil--Petersson volumes, and that the rescaled limit of the recursion of \cref{qvolrec} reproduces Mirzakhani's recursion. By induction, this proves the proposition.
\end{proof}

The $q$-deformations that we introduce here may possess structure analogous to that of their limits, namely the Weil--Petersson volumes. For example, in previous work of the authors, we prove various properties of Weil--Petersson volumes, such as the fact that $V_{g,1}^{\mathrm{WP}}(2\pi i)=0$~\cite{DNoWei}. It would be interesting to see if this has consequences for the polynomials $V_{g,1}(L)$ defined here.

\section{A \texorpdfstring{$q$}{q}-deformation of super Weil--Petersson volumes} \label{sec3}
In this section, we define a $q$-deformation of the super Weil--Petersson
volumes originally defined by Stanford and Witten~\cite{SWiJTG}.  We begin
with a brief description of the supergeometry underlying these volumes.
Roughly, a supermanifold consists of an ordinary manifold together with
additional anticommuting, or \emph{odd}, directions.  The ordinary coordinates
are called \emph{even} or \emph{bosonic}, while the odd coordinates are also
called \emph{fermionic}.  Functions in the odd coordinates take values in an
exterior algebra and are therefore nilpotent.

More precisely, a supermanifold is a locally ringed space
$\widehat{M}=(M,\co_{\widehat{M}})$ such that locally
\[
\co_{\widehat{M}}(U)
\cong
\co_M(U)\otimes\Lambda^*(V)
\]
for some vector space $V$.  The quotient of $\co_{\widehat{M}}$ by its
nilpotent ideal recovers $\co_M$, so that $M$ is the \emph{reduced space} of
$\widehat{M}$.  In this sense, $\widehat{M}$ may be viewed as an infinitesimal
thickening of $M$ by odd directions.

A basic example arises from a vector bundle $E\to M$ over a smooth manifold
$M$.  The sheaf of smooth sections of $\Lambda^*E^\vee$ defines a smooth
supermanifold
\[
\widehat{M}=(M,\Gamma(\Lambda^*E^\vee))
\]
whose odd directions are modelled on the fibres of $E$.  Integration in these
odd directions is given by Berezin integration.  If $(M,\omega)$ is
symplectic, integration over the odd directions reduces the super volume to
an ordinary integral over the reduced space,
\[
\operatorname{Vol}(\widehat{M})
=
\int_M e(E^\vee)\exp(\omega),
\]
where $e(E^\vee)$ denotes the Euler class of $E^\vee$, or an
Euler form when $M$ is non-compact.

The moduli space of genus $g$ super Riemann surfaces with $n$ marked points $\widehat{\modm}_{g,n}$ arises in this way---the odd directions are defined smoothly by the sheaf of smooth sections of $\Lambda^*E_{g,n}^\vee$, where $E_{g,n}$ is a natural vector bundle defined over the moduli space $\modm_{g,n}^{\rm spin}$ of spin curves with $n$ marked points:
\[
{\modm}_{g,n}^{\text{spin}} = \left\{ (\cc,\theta,p_1,\ldots, p_n,\phi) \,\middle|\, \phi:\theta^2\stackrel{\cong}{\longrightarrow}\omega_{\cc}^{\text{log}} \right\},
\]
where $\cc$ is an orbifold curve with isotropy subgroup $\bz_2$ at the marked points $\{p_1, \ldots, p_n\}$, known as a twisted curve~\cite{AJaMod}, and $\theta$ is a line bundle over $\cc$.



A spin structure on an orbifold curve with marked points $(\cc,p_1,\dots,p_n)$ defines a line bundle $\theta\to\cc$ that is a square root of the log-canonical bundle -- that is, $\theta^2\cong\omega_\cc^{\text{log}}$. Both $\theta$ and $\omega_\cc^{\text{log}}$ are line bundles, or orbifold line bundles, defined over $\cc$. An orbifold line bundle has a well-defined degree which may be a half-integer since the points with non-trivial isotropy can contribute one half to the degree. In particular, we have
\[
\deg\omega_{\cc}^{\text{log}}=2g-2+n \qquad \text{and} \qquad \deg\theta=g-1+\frac{1}{2}n.
\] 
Since $\deg\theta^{\vee}=1-g-\frac{1}{2}n < 0$, the bundle $\theta^{\vee}$ possesses no holomorphic sections and $h^0(\cc,\theta^{\vee})=0$. The index $h^0(\cc,\theta^{\vee})-h^1(\cc,\theta^{\vee})$ is constant over any family, so $h^0(\cc,\theta^{\vee})=0$ implies that $H^1(\cc,\theta^{\vee})$ is a vector space of constant dimension, which defines a vector bundle $E_{g,n}$ over any family of orbifold curves. Denote by $\ce$ the universal spin structure defined over the universal curve $\cu_{g,n}^{\text{spin}}\stackrel{\pi}{\longrightarrow}\modm_{g,n}^{\text{spin}}$ and define the vector bundle $E_{g,n}$ with fibre $H^1(\cc,\theta^{\vee})$ as follows.

\begin{definition} \label{obsbun}
Define the bundle $E_{g,n}:=-R\pi_*\ce^\vee \to \modm_{g,n}^{\text{spin}}$. 
\end{definition}

There are two types of behaviour of the spin structure at a marked point. The spin structure either extends over the point, in which case it is called {\em Neveu--Schwarz}, or it does not extend over the point, in which case it is called {\em Ramond}. This decomposes the moduli space of spin curves into components
\[
\modm_{g,n}^{\text{spin}}=\bigsqcup_{\sigma\in\{0,1\}^n}\modm_{g,\sigma}^{\text{spin}},
\]
where we write $\sigma=(\sigma_1, \ldots, \sigma_n)\in\{0,1\}^n$ and set $\sigma_j=1$ if $p_j$ is Neveu--Schwarz and $\sigma_j=0$ if $p_j$ is Ramond. The vector bundle $E_{g,n}$ restricts to a bundle $E_{g,\sigma}\to\modm^{\text{spin}}_{g,\sigma}\subset\modm_{g,n}^{\text{spin}}$ of rank $2g-2+\frac{1}{2}(n+|\sigma|)$, where we write $|\sigma|=\sum\sigma_j$. There is a natural Euler form $e(E_{g,\sigma})=\text{Pf}(F_A)$, defined in previous work of the second author~\cite{NorSup}. Here, $F_A$ is the curvature of the Chern connection $A$ on $E_{g,\sigma}$, defined from a natural Hermitian metric on $E_{g,\sigma}$ that comes from the super Weil--Petersson metric. The Euler form has degree $\deg e(E_{g,\sigma}^\vee)=4g-4+n+|\sigma|$.

The super Weil--Petersson volume on each component reduces to the following integral over $\modm_{g,\sigma}^{\text{spin}}$.
\begin{equation}
\widehat{V}^{\mathrm{WP}}_{g,\sigma} := 2^{g-1+\frac{1}{2}(n+|\sigma|)} \int_{\modm_{g,\sigma}^{\text{spin}}} e(E_{g,\sigma}^\vee) \, \exp\omega^{\mathrm{WP}}
\end{equation}
Let $\sigma=(1^n,0^m)$ so that $\modm_{g,\sigma}^{\text{spin}}$ parametrises spin curves with $n$ Neveu--Schwarz marked points and $m$ Ramond marked points. Consider $\omega^{\mathrm{WP}}(L_1, \ldots, L_n,0^m)$, which deforms the Weil--Petersson symplectic form only at the Neveu--Schwarz marked points. Then define
\begin{equation} \label{supvol}
\widehat{V}^{\mathrm{WP},(m)}_{g,n}(L_1, \ldots, L_n) := 2^{g-1+n+\frac{1}{2}m} \int_{{\modm}_{g,(1^n,0^m)}^{\text{spin}}}\!\! e(E_{g,(1^n,0^m)}^\vee) \, \exp\omega^{\mathrm{WP}}(L_1, \ldots, L_n),
\end{equation}
which turns out to be a polynomial in $L_1, \ldots, L_n$. Analogous to \cref{supvolpol}, define the series
\[
\widehat{V}^{\mathrm{WP}}_{g,n}(L_1, \ldots, L_n;s) := \sum_{m=0}^\infty \frac{s^m}{m!} \, \widehat{V}^{\mathrm{WP},(m)}_{g,n}(L_1, \ldots, L_n),
\]

Furthermore, define the functions
\begin{align} \label{qkerlimsup}
\widehat{H}(x,y)&:= \lim_{q\to 1} \widehat{H}_q \Big( \frac{x}{1-q}, \frac{y}{1-q} \Big) = \frac{1}{4\pi} \bigg( \frac{1}{\cosh\frac{x-y}{4}}-\frac{1}{\cosh\frac{x+y}{4}} \bigg) \notag\\
\widehat{D}(x,y,z)&:= \lim_{q\to 1} \widehat{D}_q \Big( \frac{x}{1-q}, \frac{y}{1-q}, \frac{z}{1-q} \Big) = \widehat{H}(x,y+z) \\
\widehat{R}(x,y,z)&:= \lim_{q\to 1} \widehat{R}_q \Big( \frac{x}{1-q}, \frac{y}{1-q}, \frac{z}{1-q} \Big) = \frac{1}{2} \big( \widehat{H}(x+y,z) + \widehat{H}(x-y,z) \big). \notag
\end{align}
The super Weil--Petersson volumes are uniquely determined from the initial conditions $\widehat{V}^{\mathrm{WP},(0)}_{0,1}(L) = 0$, $\widehat{V}^{\mathrm{WP},(1)}_{0,1}(L) = 0$, $\widehat{V}^{\mathrm{WP},(2)}_{0,1}(L) = 1$ and $\widehat{V}^{\mathrm{WP},(0)}_{1,1}(L_1)=\frac{1}{8}$, and the following recursion due to Stanford and Witten~\cite{SWiJTG}.

\begin{theorem}[Alexandrov and Norbury~\cite{ANoSup}]
The super Weil--Petersson volumes satisfy the following recursion for $g\geq 0$ and $n>0$.
\begin{align} 
L_1&\widehat{V}^{\mathrm{WP}}_{g,n}(L_1, \LL_K;s)=\sum_{j=2}^n\int_0^\infty \!\! x \widehat{R}(L_1,L_j,x) \, \widehat{V}^{\mathrm{WP}}_{g,n-1}(x,\LL_{K \setminus \{j\}};s) \, \dd x \notag \\
&+ \frac{1}{2} \int_0^\infty \!\! \int_0^\infty \!\! xy \widehat{D}(L_1, x, y) \Big[\widehat{V}^{\mathrm{WP}}_{g-1,n+1}(x, y, \LL_K;s) + \sum_{\substack{g_1+g_2=g \\ I \sqcup J = K}} \widehat{V}^{\mathrm{WP}}_{g_1,|I|+1}(x, \LL_I;s) \, \widehat{V}^{\mathrm{WP}}_{g_2,|J|+1}(y, \LL_J;s) \Big] \, \dd x \, \dd y. \label{volrecsup}
\end{align}
\end{theorem}

The $s=0$ case of \cref{volrecsup} was proven heuristically using supergeometry by Stanford and Witten in~\cite{SWiJTG}, and proven using algebraic geometry by the second author~\cite{NorEnu}.

We now provide the proof of \cref{super}, which is similar to that of \cref{thm:main}, but requires a new collection of $q$-series identities.

\begin{proposition} \label{prop:qidentodd}
For any $k\in\bz$,
\begin{equation} \label{qidentodd}
\sum_{m~\mathrm{odd}}(-1)^{(m-1)/2} q^{(m^2-1)/4} (1+q^m)\frac{q^{mk}}{(1-q^m)^{2k+1}}= \psi(q)^2\cdot s_k \big( \zeta^{\mathrm{odd}}_q(2), \zeta^{\mathrm{odd}}_q(4), \ldots \big),
\end{equation}
where the sum is over the odd positive integers and $s_k$ is defined as in \cref{prop:qident} by
\[
\exp \bigg(\sum_{m=1}^\infty \frac{p_m}{m} \, t^m \bigg) = \sum_{k=-\infty}^\infty s_k(p_1, p_2, \ldots) \, t^k.
\]
\end{proposition}

\begin{proof}
For $|q|<1$, the following series and product converge to holomorphic functions uniformly for $t$, in compact sets disjoint from the poles.
\begin{equation} \label{qidodd}
\sum_{m~\mathrm{odd}}(-1)^{(m-1)/2} q^{(m^2-1)/4} \frac{1+q^m}{1-q^m} \frac{(1-q^m)^2}{(1-q^m)^2+t^2q^m}=\psi(q)^2\prod_{m~\mathrm{odd}}\frac{(1-q^m)^2}{(1-q^m)^2+t^2q^m}
\end{equation}
To obtain \cref{qidodd}, put $t^2=z$ and calculate the residue at each pole in the product as follows.
\begin{align*} 
\Res_{z=-(q^{-N/2}-q^{N/2})^2}\psi(q)^2\prod_{m~\mathrm{odd}}\frac{(1-q^m)^2}{(1-q^m)^2+zq^m}&=\psi(q)^2\frac{(1-q^N)^2}{q^N} \prod_{\mathrm{odd}~m\neq N} \frac{(1-q^m)^2}{(1-q^{m-N})(1-q^{m+N})}\\
&=(-1)^{(N-1)/2} q^{(N^2-1)/4-N}(1-q^{2N}),
\end{align*} 
which uses $\big((1-q^m)^2+zq^m\big)|_{z=-(q^{-N/2}-q^{N/2})^2}=(1-q^{m-N})(1-q^{m+N})$.

The coefficient of $t^{2k}$ in the Taylor expansion about $t=0$ of \cref{qidodd} produces the $k$th identity of \cref{qidentodd} as follows. The coefficient of $t^{2k}$ in the Taylor expansion about $t=0$ of the left side of \cref{qidodd} is the sum over $m$ of the Taylor coefficient of each summand, due to the factor $q^{(m^2-1)/4}$, which guarantees convergence of the partial sums. It is given by
\[
(-1)^k\sum_{m~\mathrm{odd}}(-1)^{(m-1)/2} q^{(m^2-1)/4} (1+q^m)\frac{q^{mk}}{(1-q^m)^{2k+1}}.
\] 
Take the Taylor expansion of the logarithm of the right side of \cref{qidodd} divided by $\psi(q)^2$.
\begin{align*}
\log\prod_{m~\mathrm{odd}}\frac{(1-q^m)^2}{(1-q^m)^2+t^2q^m} &= -\sum_{m~\mathrm{odd}}\log\Big(1+\frac{t^2q^m}{(1-q^m)^2}\Big)=\sum_{k=1}^\infty (-1)^k\frac{t^{2k}}{k}\sum_{m~\mathrm{odd}} \frac{q^{mk}}{(1-q^m)^{2k}}\\
&=\sum_{k=1}^\infty (-1)^k\frac{t^{2k}}{k}\zeta^{\mathrm{odd}}_q(2k)
\end{align*}
Now exponentiate both sides and use \cref{eq:sk} to obtain
\[
\prod_{m~\mathrm{odd}}\frac{(1-q^m)^2}{(1-q^m)^2+t^2q^m}=\sum_{k=0}^\infty (-1)^k s_k(\zeta^{\mathrm{odd}}_q(2), \zeta^{\mathrm{odd}}_q(4), \ldots, \zeta^{\mathrm{odd}}_q(2k)) \, t^{2k}.
\]
Putting this together, we have \cref{qidentodd} for $k \geq 0$.

To prove \cref{qidentodd} for $k < 0$, denote by $G(t)$ the function given by \cref{qidodd}, integrate $t^{2k-1}G(t)$ over large circles, and use \cref{contint}. Writing $f_m(q) = (-1)^{(m-1)/2} q^{(m^2-1)/4} (1+q^m)$, we have
\[
\frac{1}{2\pi i} \oint_{|t|=R} t^{2k-1} \, \dd t\sum_{m~\mathrm{odd}}f_m(q)\frac{(1-q^m)}{(1-q^m)^2+t^2q^m} = \sum_{\mathrm{odd}~m\leq N}f_m(q)\frac{(1-q^m)^{2k-1}}{q^{mk}}.
\]
Again, uniform estimates show that the left side converges to 0 as $N \to \infty$. So for all $k>0$, we have
\[
\lim_{N \to \infty} \sum_{\mathrm{odd}~m\leq N}(-1)^{(m-1)/2} q^{(m^2-1)/4} (1+q^m)\frac{q^{mk}}{(1-q^m)^{2k+1}}=0,
\]
as required.
\end{proof}
Note that although the odd part of the classical zeta function is simply a scalar multiple of the Riemann zeta function, 
\[(1-2^{-2k})\zeta(2k)= \zeta^{\mathrm{odd}}(2k)=\sum_{m \text{ odd}} \frac{1}{m^{2k}}=\lim_{q \to 1} (1-q)^{2k} \, \zeta^{\mathrm{odd}}_q(2k) 
\]
there is no analogous relation between
$\zeta^{\mathrm{odd}}_q(2k)$ and $\zeta_q(2k)$, and  Proposition~\ref{prop:qidentodd} is not a consequence of Proposition~\ref{prop:qident}. 

For $\widehat{H}_q(x,y)$ defined in \cref{Hqsup}, define
\[
\widehat{F}_{2k+1}(y)=\int_0^\infty \!\! x^{2k+1}\widehat{H}_q(x,y) \, \dd x.
\]
The function $\widehat{H}_q(x,y)$ is to be understood as the series in $y^2$ given by
\[
\widehat{H}_q(x,y)=\frac{1}{4\psi(q)^2} \sinh \Big( y\frac{\dd}{\dd x} \Big) \, \widehat{h}_q(x)=\frac{1}{4\psi(q)^2}\sum_{n= 0}^\infty \frac{y^{2n+1}}{(2n+1)!} \frac{\dd^{2n+1}}{\dd x^{2n+1}} \widehat{h}_q(x),
\]
where
\[
\widehat{h}_q(x)=\sum_{m~\mathrm{odd}}(-1)^{(m+1)/2} q^{(m^2-2m-1)/4} (1+q^m)e^{\frac{x}{4}\left(q^{m/2}-q^{-m/2}\right)}.
\]

\begin{proposition} \label{prop:F2k+1odd}
For each non-negative integer $k$, we have
\[
\frac{\widehat{F}_{2k+1}(y)}{(2k+1)!} = \int_0^\infty \!\! \frac{x^{2k+1}}{(2k+1)!}\widehat{H}_q(x,y) \, \dd x = \sum_{n=0}^{k} \widehat{b}_n \frac{y^{2k+1-2n}}{(2k+1-2n)!},
\]
where $\widehat{b}_0, \widehat{b}_1, \widehat{b}_2, \ldots \in \bq[\zeta^{\mathrm{odd}}_q(2), \zeta^{\mathrm{odd}}_q(4), \ldots]$ are defined by
\[
\sum_{n=0}^\infty \widehat{b}_nz^{2n}=\exp \bigg( \sum_{m~\mathrm{odd}}\frac{\zeta^{\mathrm{odd}}_q(2m)}{m} (4z^2)^m \bigg) =1+4\zeta^{\mathrm{odd}}_q(2)z^2+8(\zeta^{\mathrm{odd}}_q(2)^2+\zeta^{\mathrm{odd}}_q(4))z^4 + \cdots.
\]
\end{proposition}

\begin{proof}
The proof of integrability of $x^{2k+1} \widehat{H}_q(x,y)$ follows the argument of the proof of \cref{F2k+1}. It is a consequence of integrability of $\widehat{h}_q(x)$, which uses the same estimates as the proof of \cref{F2k+1}, together with the inequality $\frac{1}{4\psi(q)^2} < 1$.

We begin by calculating the integral of each term of $4\psi(q)^2\widehat{H}_q(x,y)$.
\begin{align*}
 \int_0^\infty &\!\! \frac{x^{2k+1}}{(2k+1)!} \frac{1}{(2n+1)!}\frac{d^{2n+1}}{dx^{2n+1}}\widehat{h}_q(x) \, \dd x \\
 &= \sum_{m~\mathrm{odd}}(-1)^{(m-1)/2} q^{(m^2-1)/4} (1+q^m)\frac{(1-q^m)^{2n+1}}{q^{m(n+1)}4^{2n+1}}\int_0^\infty \!\! \frac{x^{2k+1}}{(2k+1)!}e^{\frac{x}{4}(q^{m/2}-q^{-m/2})} \, \dd x \\
 &=\sum_{m~\mathrm{odd}}(-1)^{(m-1)/2} q^{(m^2-1)/4} (1+q^m)\frac{4^{2k-2n+1}q^{m(k-n)}}{(1-q^m)^{2(k-n)+1}}\\
 &=4^{2k-2n+1}\psi(q)^2 \cdot s_{k-n}(\zeta^{\mathrm{odd}}_q(2), \zeta^{\mathrm{odd}}_q(4), \ldots) 
\end{align*}
where the last equality uses the $q$-series identity in \cref{qidentodd}.  In particular the integral vanishes when $n>k$ so only finitely many terms contribute to each integral of $x^{2k+1}\widehat{H}_q(x,y)$.

Hence
\[
\int_0^\infty \!\! \frac{x^{2k+1}}{(2k+1)!}\widehat{H}_q(x,y) \, \dd x =\sum_{n=0}^{k} s_n \big( \zeta^{\mathrm{odd}}_q(2), \zeta^{\mathrm{odd}}_q(4), \ldots \big)\frac{(4y)^{2k-2n+1}}{(2k-2n+1)!}.
\]
Since $\widehat{b}_n=4^ns_n(\zeta^{\mathrm{odd}}_q(2), \zeta^{\mathrm{odd}}_q(4), \ldots)$ then $\widehat{F}_{2k+1}(y)$ is a polynomial in $y$ given by
\[
\frac{\widehat{F}_{2k+1}(y)}{(2k+1)!}=\sum_{n=0}^{k+1} \widehat{b}_n \frac{y^{2k-2n+1}}{(2k-2n+1)!}. \qedhere
\]
\end{proof}

\begin{proposition} \label{symsup}
The recursion of \cref{qvolrecsup} produces series in $s^2$ with coefficients given by symmetric polynomials in $L_1^2, \ldots, L_n^2$.
\end{proposition}

\begin{proof}
The proof is similar to the proof of \cref{sym}. Integrability of the functions $z^{2k+1}\widehat{R}_q(x,y,z)$ and $y^{2i+1}z^{2j+1}\widehat{D}_q(x,y,z)$ is an immediate consequence of integrability of $x^{2k+1}\widehat{H}_q(x,y)$, as in the proof of \cref{sym}.

Furthermore, the integrals of $\widehat{D}_q(x,y,z)$ and $\widehat{R}_q(x,y,z)$ multiplied by the polynomials appearing in \cref{qvolrecsup} are uniquely determined by $\widehat{F}_{2k+1}(y)=\int_0^\infty \! x^{2k+1}\widehat{H}_q(x,y) \, \dd x$. These are polynomial due to the $q$-series identities of \cref{qidentodd}, which now replace \cref{qident}.

The kernels $\widehat{R}_q(x,y,z)$ and $\widehat{D}_q(x,y,z)$ share two properties with the kernels $R_q(x,y,z)$ and $D_q(x,y,z)$ that allow us to adapt the proof of \cref{sym} to show that $\widehat{V}_{g,n}(L_1, \ldots, L_n)$ is a symmetric polynomial. First, the vanishing $\widehat{D}_q(0,y,z)=\widehat{R}_q(0,y,z)=0$ holds. This proves divisibility of the right side of \cref{qvolrecsup} by $L_1$ and hence, gives polynomiality of $\widehat{V}_{g,n}(L_1, \ldots, L_n)$. Second, $\widehat{R}_q(-L_1,L_j,x)=-\widehat{R}_q(L_1,L_j,x)$ and $\widehat{D}_q(-x,y,z)=-\widehat{D}_q(x,y,z)$ hold, which can then be used to prove that the polynomial $\widehat{V}_{g,n}(L_1, \ldots, L_n)$ is symmetric.
\end{proof}

\begin{proposition} \label{limitsup}
The series $\widehat{V}_{g,n}(L_1, \ldots, L_n) \in\bq[[q]][L_1^2, \ldots, L_n^2][[s^2]]$ tend to the super Weil--Petersson volumes in the following rescaled $q \to 1$ limit.
\[
\lim_{q\to 1}(1-q)^{2g-2} \, \widehat{V}_{g,n}\Big(\frac{L_1}{1-q}, \ldots, \frac{L_n}{1-q}; (1-q)s \Big)=\widehat{V}^{\mathrm{WP}}_{g,n}(L_1, \ldots, L_n; s).
\]
\end{proposition}

\begin{proof}
The proof is similar to that of \cref{limit}, with the departure being the use of the following Weierstrass product 
\begin{align*}
\cos(2\pi z)=\prod_{m~\mathrm{odd}} \bigg( 1-\frac{16z^2}{m^2} \bigg) \qquad &\Rightarrow \qquad \log\cos(2\pi z)=-\sum_{k=1}^\infty \frac{\zeta^{\mathrm{odd}}(2k)}{k} (4z)^{2k} \\
&\Rightarrow \qquad \frac{1}{\cos(2\pi z)} = \exp \bigg( \sum_{k=1}^\infty \frac{\zeta^{\mathrm{odd}}(2k)}{k} \, (4z)^{2k}\bigg)
\end{align*}

The second author~\cite[Lemma 5.2]{NorEnu} introduced the function
\[
G_{2k+1}(y):=\int_0^\infty \!\! x^{2k+1}\widehat{H}(x,y) \, \dd x
\]
for $\widehat{H}(x,y)=\frac{1}{4\pi}\left(\mathrm{sech}\frac14(x-y)-\mathrm{sech}\frac14(x+y)\right)$
and showed that it satisfies
\[
\frac{G_{2k+1}(y)}{(2k+1)!} = \sum_{n=0}^{k} c_n \frac{y^{2k+1-2n}}{(2k+1-2n)!}, \qquad \text{where } \quad\frac{1}{\cos(2\pi z)}=\sum_{n=0}^\infty c_n z^{2n}.
\]

Since
\[
\lim_{q \to 1} (1-q)^{2k} \, \zeta^{\mathrm{odd}}_q(2k) = \zeta^{\mathrm{odd}}(2k) \qquad \text{and} \qquad \lim_{q \to 1} (1-q) \, \psi(q)^2 = \frac{\pi}{2},
\]
we have
\[
\lim_{q\to 1}(1-q)^{2k+2}\widehat{F}_{2k+1}\Big(\frac{y}{1-q}\Big)=G_{2k+1}(y).
\]
The polynomials $G_{2k+1}(y)$ uniquely determine the super Weil--Petersson volumes $\widehat{V}^{\mathrm{WP}}_{g,n}(L_1, \ldots, L_n;s)$, as shown by Alexandrov and the second author~\cite{ANoSup}. Similarly, the polynomials $\widehat{F}_{2k+1}(y)$ uniquely determine  $\widehat{V}_{g,n}(L_1, \ldots, L_n)$ from the recursion of \cref{qvolrecsup}, as outlined above in the proof of \cref{symsup}. One can check that the rescaled limits of the initial data match as expected and the desired result then follows by induction.
\end{proof}

\bibliographystyle{plain}
\bibliography{q-mirzakhani2}

\end{document}